\documentclass[11pt, reqno]{amsart}
\usepackage{babel,amscd, color, bbm}
\usepackage{url}
\usepackage{parskip}
\usepackage{fancyhdr}
\usepackage{vmargin}
\usepackage{appendix}
\usepackage{graphicx, amsmath, amssymb, amscd, amsthm, euscript, psfrag, amsfonts,bm,xcolor}
\usepackage{etoolbox}
\usepackage{multicol}
\patchcmd{\thebibliography}{\chapter*}{\section*}{}{}

\setmarginsrb{2 cm}{2 cm}{2 cm}{2 cm}{1 cm}{1 cm}{1 cm}{1 cm}

\newtheorem{theorem}{Theorem}[section]
\newtheorem{lemma}[theorem]{Lemma}
\newtheorem{proposition}[theorem]{Proposition}
\newtheorem{corollary}[theorem]{Corollary}\newtheorem{cor}[theorem]{Corollary}

\theoremstyle{definition}
\newtheorem{definition}[theorem]{Definition}

\theoremstyle{remark}
\newtheorem{remark}[theorem]{Remark}

\newcommand{\e}{{\varepsilon}}
\newcommand{\Z}{\mathbb{Z}}

\newcommand{\N}{\mathbb{N}}
\newcommand{\C}{\mathbb{C}}
\newcommand{\R}{\mathbb{R}}

\renewcommand{\k}{\kappa}
\newcommand{\ra}{\rightarrow}

\DeclareMathAlphabet{\mathpzc}{OT1}{pzc}{m}{it}
\newcommand{\supp}{\operatorname{supp}}
\newcommand{\op}{\operatorname}

\renewcommand{\P}{\mathcal{P}}
\newcommand{\Cal}{\mathcal}
\newcommand{\norm}[1]{\lVert #1 \rVert}
\numberwithin{equation}{section}
\newcommand{\inv}{^{-1}}

\DeclareMathAlphabet{\mathpzc}{OT1}{pzc}{m}{it}

\title{Statistics of multipliers for hyperbolic rational maps}								
\author{Richard Sharp and Anastasios Stylianou}						
\address{Mathematics Institute,
Zeeman Building,
University of Warwick,
Coventry CV4 7AL}
\email{R.J.Sharp@warwick.ac.uk}

\address{Mathematics Institute,
Zeeman Building,
University of Warwick,
Coventry CV4 7AL}
\email{Tasos.Stylianou@warwick.ac.uk}

\begin{document}
\maketitle
\begin{abstract} In this article, we consider a counting problem for orbits of hyperbolic rational maps on the Riemann sphere,
where constraints are placed on the multipliers of orbits.
Using arguments from work of Dolgopyat, we 
consider varying and potentially shrinking intervals, and
obtain a result which resembles a local central limit theorem for the logarithm of the absolute value of the multiplier and an
equidistribution theorem for the holonomies.
\end{abstract}
\section{Introduction}
A major theme in the theory of dynamical systems is the study of the distribution of periodic orbits.
This is particularly well-developed for hyperbolic systems, where one finds precise asymptotics and equidistribution results. Here, equidistribution can refer to spatial results, where averages
of orbital measures converge to a prescribed limiting measure, or to equidistribution with respect to
some symmetry of the system. This paper can be seen in the context of the second setting. 
A major advance in this theory came from the work of Dolgopyat \cite{dol}, which introduced an 
approach to obtaining more precise results.

Let us now be more precise about our setting. Let $f : J \to J$ be a hyperbolic rational map 
restricted to its Julia set and $0<\delta<2$ to be the Hausdorff dimension of $J$.
(See the next section for formal definitions.) A periodic orbit
$\tau = \left\{z,f(z),\ldots,f^{n-1}(z)\right\}$ (with $f^n(z)=z$) is called {\it primitive} if $f^m(z) \ne z$ for all
$1 \le m <n$. We denote the set of primitive periodic orbits by $\mathcal P$.
For each
$\tau = \{z,f(z),\ldots,f^{n-1}(z)\} \in \mathcal P$, we define its {\it multiplier} 
\[
\lambda(\tau) := (f^n)'(z)
\]
and its {\it holonomy} 
\[
\hat{\lambda}(\tau) : = \frac{\lambda(\tau)}{|\lambda(\tau)|}
\in \mathbb S^1,
\]
where $\mathbb S^1$ denotes the unit circle in $\mathbb C$.
A beautiful recent result of Oh and Winter \cite{hodw} states that, apart from a small set of
completely classified exceptional cases,
there exists $\varepsilon > 0$ such that 
\begin{equation}\label{ow_pot}
\# \{\tau\in \P: | \lambda(\tau) | <t\}  =  \op{Li}(t^{\delta})+O(t^{\delta -\e})
\end{equation}
and, for any $\psi\in C^4(\mathbb S^1)$,
\begin{equation*} \sum_{\tau\in \P : |\lambda(\tau) | <t} \psi \big(\hat \lambda(\tau)\big) 
=  \left(\int_{0}^{1} \psi (e^{2\pi i \theta})\; d\theta\right) \op{Li}(t^{\delta}) +O(t^{\delta -\e}),
\end{equation*}
as $t \to \infty$.
Here, $\mathrm{Li}$ denotes the logarithmic integral $\mathrm{Li}(x) = \int_2^x (\log u)^{-1} \, du 
\sim x/\log x$, as $x \to \infty$ and we write $f(x)=O(g(x))$ as $x\rightarrow\infty$ whenever there exists $C>0$ and $x_0\in\R$ such that for all $x\geq x_0$ we have that $|f(x)|\leq Cg(x)$. We also write $f(x)\sim g(x)$ as $x\rightarrow\infty$ whenever $\lim_{x\rightarrow\infty}{f(x)}/{g(x)}=1$.

In this paper, we take a slightly different viewpoint.
Instead of counting $\tau = \{z,f(z),\ldots,f^{n-1}(z)\}$ according to the modulus of its multiplier 
$|\lambda(\tau)|$, we count by the period $|\tau|=n$ but impose constraints on 
deviations of $\log |\lambda(\tau)|$
from the period and on the holonomy. More precisely, for $\alpha \in \mathbb R$, an interval
$I \subset \mathbb R$ and an arc $S \subset \mathbb S^1$, 
and writing $\mathcal P_n = \{\tau \in \mathcal P \hbox{ : } |\tau|=n\}$,
we aim to study the behaviour of
\[
\pi(n,\alpha,I,S) :=\#\{\tau \in \mathcal P_n \hbox{ : } \log |\lambda(\tau)| - n\alpha \in I \mbox{ and } \hat{\lambda}(\tau)
\in S\},
\]
as $n \to \infty$. We need to impose a restriction on $\alpha$ and, to do this, define the closed interval
\[
\mathcal I_f := \left\{\int \log |f'| \, d\mu \hbox{ : } \mu \in \mathcal M_f\right\},
\]
where $\mathcal M_f$ is the set of $f$-invariant probability measures on $J$. We also asume that the Julia set of $f$ is not contained in a circle in $\widehat\C$ since otherwise all holonomies are real. We write $\ell$ for Lebesgue measure on $\mathbb R$ and $\nu$ for the normalised Haar measure on 
$\mathbb S^1$.

\begin{theorem}
Let $f : J \rightarrow J$ be a hyperbolic rational map of degree $d \ge 2$ restricted to its Julia set
such that $J$ is not contained in a circle in $\widehat \C$. Then, 
for $\alpha \in \mathrm{int}(\mathcal I_f)$, there exists $\sigma_\alpha>0$ and 
$\xi_\alpha \in \mathbb R$ such that
\begin{align*}
\pi(n,\alpha,I,S) \sim \frac{\nu(S)}{\sigma_\alpha \sqrt{2\pi}}
\int_I e^{-\xi_\alpha x} \, dx \; \frac{e^{H(\alpha)n}}{n^{3/2}}, 
\qquad \text{ as $n \to \infty$},
\end{align*}
where
\[
H(\alpha) = \sup\left\{h_f(\mu) \hbox{ : } \mu \in \mathcal M_f \text{ and } \int \log |f'| \, d\mu =\alpha\right\}.
\]
In particular, if $\alpha=\int \log|f'|\,d\mu_{\max}$, where $\mu_{\max}$ is the measure of maximal entropy then
\begin{align*}
    \pi(n,\alpha,I,S)\sim\frac{\nu(S)\ell(I)}{\sigma_\alpha\sqrt{2\pi}} 
    \frac{d^n}{n^{3/2}} \,\,\,\,\,\text{  as $n\rightarrow\infty$.    }
\end{align*}
\end{theorem}

We can also allow $I$ and $S$ to shrink at suitably slow rates as 
$n$ increases. The corresponding result will appear below as Theorem \ref{main}.

\section{Hyperbolic rational maps}

Let $f:\widehat{\mathbb{C}}\rightarrow\widehat{\mathbb{C}}$ be a rational map of degree 
$d \ge 2$. 
Recalling the definitions in the introduction,
a periodic orbit can be classified as \textit{repelling, attracting} or \textit{indifferent} depending on whether its multiplier has modulus greater than, less than, or equal to one, respectively. Then, the Julia set of $f$ is defined as the closure of the union of repelling periodic orbits and denoted by $J=J(f)$. It is a compact $f^{\pm1}$-invariant subset of $\widehat\C$ and the reader is referred to Milnor’s classical text \cite{mil} for an excellent and systematic introduction to the dynamics of functions of one complex variable. In particular, we note that such a map has topological entropy $h(f)=\log d$
and $\#\{z \in \mathbb C : f^n(z)=z\}=d^n$.

We say that a rational map $f:\widehat{\C}\rightarrow\widehat{\mathbb{C}}$ is \textit{hyperbolic} if  $f$ is eventually expanding on $J$, that is there exist
constants $c>0$ and $\gamma>1$ such that
\begin{align}
     |(f^n)'(z)| \ge c\gamma^n    \label{hyperbol}
\end{align}
for all $z \in J$ and all $n \ge 1$.


For such a map, it is known that at most $2d-2$ primitive periodic orbits are not repelling. Therefore, to study asymptotic counting problems for periodic orbits of $f$ we can focus, without any loss of generality, to the study of the repelling periodic orbits. 
We write $\delta$  for the Hausdorff dimension of $J$;
this satisfies $0<\delta<2$ \cite{sullivan}.
We will impose an additional hypothesis on $f$: we suppose that $J$ is not contained in any circle
in $\widehat{\mathbb C}$. In particular, this implies that $f$ is not conjugate by a M\"obius transformation 
to a monomial $z \mapsto z^{\pm d}$ for any $d\in\N$.

We will now give a more precise version of our results.
As in the introduction,
$\mathcal M_f$ is the set of $f$-invariant probability measures on $J$, which is 
convex and compact with respect to the weak$^*$ topology.
Hence, the image of $\mathcal M_f$ onto the reals under the continuous projection \[\mu\mapsto \int_J \log|f'|\,d\mu\] is an interval, which we denote by $\mathcal I_f$. 
Since we are assuming that $f$ is not M\"obius conjugate to a  monomial, $\mathcal I_f$ has 
non-empty interior.
(If $\mathcal I_f$ is a single point then $\log |f'|$ is cohomologous to a constant, which is incompatible with
(\ref{ow_pot}).)
We define
\[
H(\alpha):=\sup\left\{h_f(\mu)\,:\, \mu \in \mathcal M_f \text{ with }\int\log|f'|\,d\mu=\alpha \right\},
\]
where $h_f(\mu)$ denotes the measure-theoretic entropy.
There is a unique $\mu_\alpha \in \mathcal M_f$ that realises this supremum above and a unique
$\xi_\alpha \in \mathbb R$ such that
\[
h_f(\mu_\alpha) + \xi_\alpha \int \log |f'| \, d\mu_\alpha
= \sup\left\{h_f(\mu) + \xi_\alpha \int \log |f'| \, d\mu \hbox{ : } \mu \in \mathcal M_f\right\}.
\]
We also define the variance of $\log |f'| -\alpha$ by 
\[\sigma^2_\alpha:=\lim_{n\rightarrow\infty}\frac{1}{n}\int 
\big( \log |(f^n)'|-n\alpha \big)^2 \,d\mu_\alpha.\]
Our hypothesis on $f$ implies that $\sigma_\alpha^2>0$.
These statements will be proved in the next section.

We want to consider the quantity $\pi(n,\alpha,I,S)$ defined in the introduction.
However, we also wish to consider a situation where $I$ 
and $S$ shrink as $n \to \infty$.
To do this, let $K \subset \R$ be a compact set, let 
$(I_n)_{n=1}^\infty$ be a sequence of intervals contained 
in $K$ and let $(S_n)_{n=1}^\infty$ be a sequence of arcs 
in $\mathbb{S}^1$.
 We are mainly interested in the two special cases where the sequences
$(I_n)_{n=1}^\infty$ and $(S_n)_{n=1}^\infty$
are constant, corresponding to the case of a fixed interval and a fixed arc as in the introduction, and where the sequences $(\ell(I_n))_{n=1}^\infty$ and $(\nu(S_n))_{n=1}^\infty$ tend to zero, hence realising shrinking intervals. Similar asymptotic counting problems were considered in \cite{rcor}, \cite{ps-erg} and \cite{ps-man}. 
 
We say that a sequence $(s_n)_{n=1}^{\infty}$ has sub-exponential growth if $\limsup_{n\rightarrow\infty}|\log s_n|/n=0$. We have the following theorem.

\begin{theorem}\label{main}
Let $f : \widehat{\mathbb{C}} \rightarrow \widehat{\mathbb{C}}$ be a hyperbolic rational map of degree at least $2$ such that its Julia set is not contained in a circle in $\widehat \C$. 
Let $K \subset \R$ be a compact set, 
let $(I_n)_{n=1}^\infty$ be a sequence of intervals 
in $K$ and let $(S_n)_{n=1}^\infty$ be a sequence of 
arcs in $\mathbb{S}^1$. Furthermore, suppose that
$(\ell(I_n)^{-1})_{n=1}^{\infty}$ and $(\nu(S_n)^{-1})_{n=1}^{\infty}$ have sub-exponential growth. Then, for each $\alpha \in \mathrm{int}(\mathcal I_f)$, we have that
    \begin{align}\label{main1}
        \pi(n,\alpha,I_n,S_n)\sim\frac{\nu(S_n)}{\sigma_\alpha\sqrt{2\pi}}\int_{I_n} e^{-\xi_\alpha x} \, dx \; \frac{ e^{H(\alpha)n}}{n^{3/2}},\,\,\,\,\,\text{  as $n\rightarrow\infty$.     }
    \end{align}
    In particular, if in addition we have that 
    $\lim_{n\rightarrow\infty}\ell(I_n)=0$ and  $p_n\in I_n$ 
    then
    \begin{align}\label{main2}
       \pi(n,\alpha,I_n,S_n)\sim \frac{\nu(S_n)\ell(I_n)e^{-\xi_\alpha p_n}}{\sigma_\alpha\sqrt{2\pi }} \;\frac{e^{H(\alpha)n}}{n^{3/2}},\,\,\,\,\,\text{  as $n\rightarrow\infty$.     }
  \end{align} 
    \end{theorem}
    
\begin{corollary}    
If $\alpha=\int \log|f'|\,d\mu_{\max}$, where $\mu_{\max}$ is the measure of maximal entropy then
\begin{align}
    \pi(n,\alpha,I_n,S_n)\sim\frac{\nu (S_n)\ell(I_n)}{\sigma_\alpha\sqrt{2\pi }} \;\frac{d^{n}}{n^{3/2}}, \,\,\,\,\,\text{  as $n\rightarrow\infty$.    } \label{main3}
\end{align}
\end{corollary}
\section{Thermodynamic formalism for hyperbolic rational maps}
The main purpose of this section is to describe how one can study the dynamics of a hyperbolic rational map using transfer operators and to obtain some decay estimates for them. We begin by recalling the essential features of this approach but for more details the reader is referred to \cite{rueexp}. 
We fix a hyperbolic rational map $f$ of degree $d\geq2$. 
Further, we assume that the Julia set of $f$ is not contained inside a circle in $\widehat \C$.

\subsection{Markov Partitions}
For any small $\e>0$,
we can find a Markov partition for $J$:
compact subsets $P_1, \ldots , P_{N}$ of $J$
each of diameter at most $\e$, such that
\begin{enumerate}
\item $J = \bigcup_{i=1}^{N} P_i$,
\item 
$\overline{ \mbox{int}(P_i)} = P_i$,  $i=1,\ldots,N$,
 \item 
 $\mbox{int}(P_i) \cap \mbox{int}(P_j)= \emptyset$, whenever $i\neq j$,
\item for each $i=1,\ldots, N$, 
$f(P_i) = \bigcup_{j \in \mathcal N_i} P_j$,
where $\mathcal N_i =\{j \in \{1,\ldots,N\}
\hbox{ : } f(P_i) \cap \mathrm{int}(P_j) \ne \emptyset\}$
\end{enumerate}
(where closure and interior is taken relative to $J$).
 
Given a Markov partition $P_1, \ldots, P_{N}$, 
we can find open neighbourhoods
$U_j \supset P_j$ such that: 
\begin{enumerate}
 \item $f$ is injective on the closure of each $U_j$
 and
 on the union $U_i\cup U_j$, whenever $U_i \cap U_j \neq \emptyset$,
 \item each $P_i$ is {\it not} contained in $ \bigcup_{j\neq i} \overline{U_j}$,
 \item for each pair $i, j$ with $f(P_i) \supset P_j$ 
 there is a local inverse $g_{ij} : U_j \rightarrow U_i$ for $f$.
\end{enumerate}   
We write $U = \coprod_{i=1}^N U_i$ for the {\em disjoint}  union of the neighbourhoods $U_i$.

The structure of the partition allows us to define an 
$N\times N$ matrix $M$ with zero-one entries, where
\[ M_{i j} = \left\lbrace \begin{array}{l} 1 \mbox{ if } f(P_i) \supset P_j,  \\ 0 \mbox{ otherwise. } \end{array} \right. \]
\subsection{Ruelle Transfer Operators and the Pressure Function} \label{Definetransferoperatorshere} 

By the hyperbolicity assumption the Julia set of $f$, and hence $U$, does not contain any critical points, that is points where the derivative of $f$ vanishes. We can therefore define the following real analytic functions related to $f$, which will help us in the study of multipliers and holonomies of periodic orbits.
\begin{definition}\label{dist}We define the \textit{distortion function}
\[r(z) = \log |f'(z)|\]
and the \textit{rotation function}
\[\theta(z) = \arg(f'(z)) \in \mathbb{R} / 2\pi\mathbb{Z},\]
which are both defined on $U$. 
\end{definition} 

For a function $w : U \to \mathbb R$ (or $\mathbb C$)
and $n \ge 1$, we write
\[
w^n(z) = \sum_{j = 0}^{n-1} w(f^{j}(z)).
\]
(The context should make clear that this is not
an iterate.)
Hence, when $\tau =\{z,f(z),\ldots,f^{n-1}(z)\} \in\mathcal P_n$, we have $\lambda(\tau)=
(f^n)'(z)=e^{r^n(z)+ i\theta^n(z)}$.

We proceed to define the Ruelle transfer operators as well as recalling some concepts from thermodynamic formalism. Write $C^1(U)$ for functions in $C^1(U, \C)$ with bounded derivatives. Then, for $F\in C^1(U)$, we define the transfer operator
$\mathcal{L}_{F} : C^1(U) \rightarrow C^1(U)$
by
\begin{eqnarray*}( \mathcal{L}_{F}\,w) (x) &:=& \sum_{i\, :
\, M_{ij} = 1 } e^{F(g_{ij}x)}w(g_{ij}x)  \mbox{ when } x \in U_j .\end{eqnarray*}
Furthermore, we define a family of modified $C^1$ norms on $C^1(U)$ by 
\[ \left\|w\right\|_{(t)} := \begin{cases} \left\|w\right\|_\infty + \frac{\left\| w' \right\|_{\infty}}{t}& \text{ if $t\ge 1,$}
\\  \left\|w\right\|_\infty + {\left\| w' \right\|_{\infty}} &\text{if $0<t<1$}.\end{cases} \]
The reason for this is that at the end of this section we will encounter a family of transfer operators which is not uniformly bounded using the usual $C^1$ norm. However, these modified norms $\left\| \cdot \right\|_{(t)}$ will help us find sufficiently good bounds for large values of $t$.

\begin{definition} Given a continuous function $g:J\rightarrow\R$ we define the topological pressure of $g$ by
\[P(g) := \sup\ \left\{ h_f( \mu) + \int g\, d\mu
\hbox{ : } \mu \in \mathcal M_f\right\}.\]
Moreover, we call $\mu$ an equilibrium state of 
$g$ if $P(g)=h_f( \mu) + \int g\, d\mu$.
\end{definition}
If $g$ is a H\"older continuous
function then it has a unique equilibrium state, which is fully supported and
ergodic; we denote this by $m_g$. Given two functions $g,h$ we have the inequality
\begin{align}\label{presnorm}
    \left| P(g)-P(h) \right| \leq \norm{g-h}_{\infty}
\end{align}
Two continuous functions $g$ and $h$ are called cohomologous if there exists a continuous function $u:J\rightarrow\R$ such that $g-h=u\circ f-u$. 
If $g$ and $h$ are H\"older continuous then $m_g=m_h$ 
if and only if $g-h$ is cohomologous to a constant.
If $g$ and $h$ are H\"older continuous then the function $\mathbb R \to \mathbb R: 
t\mapsto P(tg+h)$ is real analytic and
\begin{align}
    \frac{dP(tg+h)}{dt}\bigg|_{t=0}&=\int g\,dm_h ,\label{presder}\\
    \frac{d^2P(tg+h)}{dt^2}\bigg|_{t=0}&=\lim_{n\rightarrow\infty}\frac{1}{n} \int \left( g^n(x)-n\int g\,dm_h \right)^2\,dm_h,
\end{align}
see \cite{rue2,papo}.
Furthermore, if $g$ is not cohomologous to a constant
then $t\mapsto P(tg+h)$ is strictly convex and
\begin{equation}
\frac{d^2P(tg+h)}{dt^2}\bigg|_{t=0}>0.
\end{equation}

We will now prove some of the statements made in the previous section. We have the following result.

\begin{lemma}\label{equil}
For each $\alpha \in \mathrm{int}(\mathcal I_f)$, there is a unique $\xi=\xi(\alpha) \in \mathbb R$
such that
$
H(\alpha) = h_f(m_{\xi r})
$
and 
\[
\int r \, dm_{\xi r} = \alpha.
\]
\end{lemma}

\begin{proof}
Since the Julia set of $f$ is not contained in a circle 
in $\widehat\C$, then $f$ is not conjugate to a monomial and so the distortion function $r$ is not cohomologous to a constant. 
Therefore, the function $\mathfrak p : \mathbb R \to \mathbb R$ defined by
$\mathfrak p(t) =P(tr)$ is strictly convex. 

Now
consider the set 
\[
\mathcal D:=\{\mathfrak p'(\xi) \hbox{ : } \xi \in \mathbb R\}
= \left\{\int r \, dm_{\xi r} \hbox{ : } \xi \in \mathbb R
\right\} \subset \mathcal I_f.
\]
Since $\mathfrak p$ is strictly
convex, $\mathcal D$ is an open interval.
By the definition of pressure, for all $\mu \in \mathcal 
M_f$,
\[
\mathfrak p(t) \ge h_f(\mu) + t\int r  \, d\mu.
\]
In particular, the graph of the convex function 
$\mathfrak p$ lies above a line with slope $\int r \, d\mu$  (possibly touching it tangentially) and so 
$\int r \, d\mu \in \overline{\mathcal D}$. Thus, since $\mu$ is arbitrary, $\mathrm{int}(\mathcal I_f) \subset
\overline{\mathcal D}$, and so we have
$\mathcal D = \mathrm{int}(\mathcal I_f)$.
Thus, for $\alpha \in \mathrm{int}(\mathcal I_f)$, there is a unique $\xi=\xi(\alpha) \in \mathbb R$ with
\[
\alpha = \mathfrak p'(\xi) = \int r \, dm_{\xi r}.
\]

Since the map $\mu \mapsto h_f(\mu)$ is upper semi-continuous \cite{lyubich}, the supremum in
\[
H(\alpha) = 
\sup\left\{h_f(\mu) \hbox{ : } \int r \, d\mu = \alpha\right\}
\]
is attained.
Since $m_{\xi r}$ is the equilibrium state for $\xi r$, 
we have, for any $\mu \in \mathcal M_f$ with 
$\mu \ne m_{\xi r}$,
\[
h_f(m_{\xi r}) + \xi \int r \, dm_{\xi r} 
>
h_f(\mu) + \xi \int r \, d\mu.
\]
In particular, if $\int r \, d\mu=\alpha$ then
$h_f(m_{\xi r}) > h_f(\mu)$. Therefore, 
$m_{\xi r}$ is the unique measure 
with the desired properties.
\end{proof}

\begin{remark}\label{intro_non-lattice}
Above, we used that $r$ is not cohomologous to a constant.
In fact, Oh and Winter proved a stronger statement (\cite[Corollary 6.2]{hodw}) that if $J$ is not contained in a circle then
$r$ satisfies the non-lattice property, i.e. that it is not cohomologous to any function of the form 
$a+bu$, with $a,b \in \R$ and
$u:J\rightarrow \Z$.
\end{remark}

Setting $\mu_\alpha = m_{\xi_\alpha r}$,
we have the measure 
whose existence is claimed in section 2.
Furthermore,
\[
\sigma_\alpha^2 = \lim_{n \to \infty}
\frac{1}{n} \int (r^n -n\alpha)^2 \, 
d\mu_\alpha = \mathfrak p''(\xi) >0,
\]
where we have used that $m_{\xi r} = m_{\xi(r-\alpha)}$.

For the rest of the paper, we will fix 
$\alpha\in\mathrm{int}(\Cal I_f)$ and set $\xi=\xi_\alpha$ as in Lemma \ref{equil}. We will also write $R:=r-\alpha$
and $R^n(x):=r^n(x)-n\alpha$ and note that,
by Lemma \ref{equil}, we have that 
$H(\alpha)=P\left(\xi R\right)$.

We will need to consider the function $s\mapsto e^{P(sR)}$, $s\in\R$. This function is real analytic and has an analytic extension in a neighborhood of the real line. The following lemma will prove useful in our analysis. 

\begin{lemma}\label{pressure}\cite[Proposition 4.7]{papo}
The function $t\mapsto e^{P\left((\xi+it)R\right)}$ is analytic and for some $\e>0$ we can write for each $t\in[-\e,\e]$
\[
e^{P\left((\xi+it)R\right)}=e^{P(\xi R)}\left(1-\frac{\sigma_\alpha^2 t^2}{2}+O(|t|^3)\right)
\]
where the implied constant is uniform on $[-\varepsilon,\varepsilon]$.
\end{lemma}
\subsection{Decay Estimates}
The approach in this subsection is motivated by Dolgopyat's seminal work on exponential mixing of Anosov flows in \cite{dol}. This was later used by 
Pollicott and Sharp to obtain an analogue of the prime number theorem with an exponential error term for closed geodesics on a compact negatively curved surface \cite{rexp}.
Naud adapted Dolgopyat's analysis to prove a similar result for convex co-compact surfaces \cite{Na}
as well as Oh and Winter whose work was in the current setting of expanding rational maps
\cite{hodw}. We use a similar approach to obtain bounds on the spectral radii of a family of transfer operators in order to extract our asymptotic result in the final section. 

Recall $f : \widehat{\mathbb{C}} \rightarrow \widehat{\mathbb{C}}$ is a hyperbolic rational map of degree at least $2$ and $\alpha,\, \xi$ are fixed constants as in Lemma \ref{equil}. We now consider the family of Ruelle transfer operators 
$\mathcal L_{(\xi+ib,k)}$, for $b\in\R$ and  $k\in\Z$, 
where $\mathcal L_{(s,k)} := \mathcal L_{sR+ik\theta}$.

We recall the following theorem.
\begin{theorem}[Ruelle--Perron--Frobenius Theorem, \cite{rueexp}]\label{thm_RPF}
Let $u \in C^1(U)$ be real valued. Then
\begin{itemize}
	    \item the operator $\mathcal{L}_{u}$ has a simple maximal positive eigenvalue $\lambda = e^{P(u)}$ with an associated strictly positive eigenfunction $\psi\in C^1(J)$,
	    \item the rest of the spectrum is contained in a disk of radius strictly smaller than $e^{P(u)}$ and
	    \item there is a unique probability measure $\mu$ on $J$ such that $\mathcal{L}_{u}^*\mu = e^{P(u)}\mu$ and $\int \psi\, d\mu =1$.
\end{itemize} 
		If $v \in C^1(U)$ is real valued then
		the spectral radius of $\mathcal{L}_{u+iv}$ is bounded above by $e^{P(u)}$.
	    
\end{theorem}	    




By Theorem \ref{thm_RPF},
the spectral radius of $\mathcal L_{(\xi+ib,k)}$ is bounded above by $e^{P(\xi R)}$. The aim of this subsection is to show that in fact, when the Julia set of $f$ is not contained in a circle in $\widehat \C$ we can bound the spectral radius of $\mathcal L_{(\xi+ib,k)}$ away from $e^{P(\xi R)}$ when $(b,k) \ne (0,0)$. To achieve this we fix arbitrary $b\in\R$ and $k\in\Z$ and consider the transfer operator $\mathcal L_{(\xi+ib,k)}$.

As a first step we want to consider a normalised transfer operator and we thus add a coboundary and a constant to $\xi R$. Write $v=b R+ k\theta$ and 
\[
u=\xi R+\log\psi-\log\psi\circ f-P(\xi R)
\]
with $\psi$ the 
positive eigenfunction of $\mathcal{L}_{(\xi,0)}$ with corresponding eigenvalue $e^{P(\xi R)}$
guaranteed by Theorem \ref{thm_RPF}. We then get that
\begin{align}\label{normal}
\mathcal L_{u}\mathbbm 1 =\mathbbm 1 \,\text{   and   }\,\mathcal L_{(\xi+ib,k)}=e^{P(\xi R)} M \mathcal L_{u+iv} M^{-1}
\end{align}
where $M$ is the multiplication operator by $\psi$. Thus, to show that the spectral radius of $\mathcal L_{(\xi+ib,k)}$ is less than $e^{P(\xi R)}$ it suffices to show that the spectral radius of $\mathcal L_{u+iv}$ is less than $1$. Write $F_{b,k}=u+iv$. Below we show that the spectral radius of our operator, denoted by $\op{spr}(\mathcal L_{F_{b,k}})$, is strictly less than $1.$ 

Let $\mu$ be the unique probability measure on $J$
satisfying $\mathcal L_{u}^*\mu=\mu$, as guaranteed by 
Theorem \ref{thm_RPF}. We regard $\mu$ as a measure on $U$ by taking $\mu = \sum \mu_j$ where $\mu_j$ is the restriction of $\mu$ to the copy of $P_j$ sitting inside $U_j$. Since the boundary points of $\mu_j$, that is points in $U_j\setminus P_j$, have zero mass, $\mu$ is a probability measure on $U$.
\begin{definition}\label{doubling}
We say that a probability measure $m$ on $J$ has the doubling property if there exists a positive constant $C$ such that for all $x\in J$ and all $\e>0$ we have that \[m(B(x,2\e))\leq C \cdot m(B(x,\e)).\]
\end{definition}
We know that in fact $\mu$ is a doubling measure (see \cite[Theorem 8]{pes}). Moreover, as in \cite[Proposition 4.5]{hodw}, it follows that the restrictions $\mu_j$ satisfy the doubling property as probability measures on $P_j$. We therefore have all the properties required to get the following theorem.
\begin{theorem}[Theorem 2.7, \cite{hodw}] \label{decayestimate} 
Suppose that the Julia set of $f$ is not contained in a circle in $\widehat{\mathbb C}$. 
Then there exist  $C > 0$ and $\rho \in (0, 1)$ such that for any $w\in C^1(U)$ with $\left\|w\right\|_{(|b|+|k|)}\leq1$ and any $n\in \N$
\[  \left\| \mathcal{L}^n_{F_{b,k}}w\right\|_{L^2(\mu)} \leq C  \rho^{n}, \]
whenever $|b|+|k|\geq 1$. 
\end{theorem} 
Using a standard argument (see \cite{dol,Na}) we can convert the bounds on the $\norm{\cdot}_{L^2(\nu)}$ norm to bounds for the modified $\norm{\cdot}_{(t)}$ norm. Then noting that $\norm{\cdot}_{C^1}\leq (|b|+|k|)\norm{\cdot}_{(|b|+|k|)}$ for $|b|+|k|\geq1$ we get the following corollary.
\begin{cor}\label{L2}
Suppose that the Julia set of $f$ is not contained in a circle in $\widehat \C$. Then, for any $\e>0$, there exist $C_\e>0$ and $\rho_\e\in(0,1)$ such that for all $b\in\mathbb R$ and all $k\in\mathbb Z$ with $|b|+|k|>1$ we have that
\[
\|\mathcal L_{F_{b,k}}^n\|_{C^1}\le C_\e  ( |b| + |k| )^{1+\e} \rho_\e^n,
\]
for all $n\in\N$. In particular, $\op{spr}(\mathcal L_{F_{b,k}})<\rho_\e<1.$

\end{cor}

Now that we have established the required bounds on the $C^1$ norm of our transfer operators we proceed to bound the sums 
\[
Z_n(s,k):=\sum_{f^nx=x} e^{sR^n(x)+ ik \theta^n(x)} 
\]
for $s=\xi+ib$. This next result follows essentially from Ruelle's work in \cite{rue}, except that we require explicit 
dependence on $b$ and $k$. A proof can be found in the appendix 
of \cite{Na} without the dependence on $k\in\Z$, which appeared 
as Proposition 6.1 in \cite{hodw}. In the statement below,
$\chi_j$ is the characteristic function of $U_j$ for each $1\le j\le m$.
(Note that, 
since $U$ is the {\it disjoint} union of the sets $U_j$, for each such $j$ we have that $\chi_j\in C^1(U)$.)
\begin{proposition}\label{zn} Fix an arbitrary $b_0>0$. There exists $x_j\in P_j$, for $1\leq j\leq m$, such that for any $\eta>0$, there exists $C_\eta>0$ such that for all $n\ge 2$ and any $k\in\mathbb Z$
\[ \left| Z_n(\xi+ib,k)-\sum_{j=1}^{N} \mathcal L_{(\xi+ib,k)}^n (\chi_j)(x_j) \right| \le 
\\ C_\eta (|b|+|k|)   \sum_{p=2}^n \|\mathcal L_{(\xi+ib,k)}^{n-p}\|_{C^1} \left(\gamma^{-1} e^{\eta +P(\xi R)}\right)^p\]
for all $|b|+|k|>b_0$.
\end{proposition}

We are now ready to prove the decay estimates that will give us the proof of Theorem \ref{main} in the next section. Fixing $\e>0$ then by Corollary \ref{L2} and Proposition \ref{zn}, we get that for all $|b|+|k|>1$,
\begin{align*}
   \left| Z_n(\xi+ib,k)\right| &\leq \left|  Z_n(\xi+ib,k)  - \sum_{j=1}^{N} \mathcal L_{(\xi+ib,k)}^n (\chi_j)(x_j)\right|+ N\,C_{\e} (|b|+|k|)^{1+\e}\left(\rho_\e e^{P(\xi)}\right)^n\\
   &\leq C_\eta C_\e (|b|+|k|)^{2+\e} \left(\rho_\e e^{P(\xi)}\right)^n  \sum_{p=2}^n \left(\frac{e^{\eta}}{\gamma \rho_\e} \right)^p + N\,C_{\e} (|b|+|k|)^{1+\e}\left(\rho_\e e^{P(\xi)}\right)^n
\end{align*} 
We note that it is possible to choose $1>\rho_\e>1/\gamma$ (recall that $\gamma$ is the expansion rate given in (\ref{hyperbol})). Provided $\eta$ is small enough such that $e^{\eta}/\gamma \rho_\e<1$ we get that for some $C>0$
\begin{align}\label{zndecay}
   \left| Z_n(\xi+ib,k)\right| \leq C \left(|b|+|k|\right)^{2+\e} \left(\rho_\e e^{P(\xi R)}\right)^n.
\end{align}

Finally, we will also need a more elementary result to bound the sums $Z_n(\xi+ib,0)$ for small $b\in \R$. These estimates can be derived as in the symbolic case in \cite{papo}. 
\begin{lemma}\label{decay}
Let $K\subset\R$ be a compact set. There exists $\varepsilon>0$ such that for each $n\in\mathbb N$ and some $\beta\in(0,1)$ we have that
\begin{enumerate}
        \item for $b\in K\setminus(-\varepsilon,\varepsilon)$ we can bound $Z_n(\xi+ib,0)=O(\beta^ne^{H(\alpha)n})$ and
        \item for $b\in(-\varepsilon,\varepsilon)$ we have\[
    Z_n(\xi+ib,0) =e^{nP((\xi+ib)R)}+O(\beta^{n}e^{H(\alpha)n}).
    \]
\end{enumerate}
\end{lemma}

\begin{proof}
For part $(1)$, we use the fact that, since $R$ is non-lattice 
(see Remark \ref{intro_non-lattice}) we have that $\op{spr}(\mathcal L_{(\xi+ib,0)})<e^{P(\xi R)}$ for $b \ne 0$,
with a uniform bound on $K \setminus (-\varepsilon,\varepsilon)$,
and Proposition \ref{zn}. 
Part $(2)$ follows from the spectral gap in the Ruelle--Perron--Frobenius theorem, which is uniform over 
an interval $(-\varepsilon,\varepsilon)$. 
\end{proof}

\section{Proof of Theorem \ref{main}} 
Throughout this section we fix a hyperbolic rational map $f : \widehat{\mathbb{C}} \rightarrow \widehat{\mathbb{C}}$ of degree at least $2$. We suppose that its Julia set is not contained inside a circle in $\widehat \C$ and we fix $\alpha$ a constant in the interior of $\mathcal I_f$. We set $\xi=\xi(\alpha)$ to be the unique real number given by Lemma \ref{equil}. Let $K \subset \R$ be a compact set, 
let $(I_n)_{n=1}^\infty$ be a sequence of intervals 
in $K$ and let $(S_n)_{n=1}^\infty$ be a sequence of arcs in $\mathbb{S}^1$. For convenience we parametirise $\mathbb \R/\Z$ as $\left[-\frac{1}{2},\frac{1}{2}\right]$ and assume that the sequence of arcs $(S_n)_{n=1}^\infty$ is contained inside a fixed reference arc $S=\left[-\frac{\k}{2},\frac{\k}{2}\right]$ of length $\k<1$.

For each $n\in\N$ we denote by $p_n$ the midpoint of the interval $I_n$ and by $\vartheta_n$ the midpoint of the arc $S_n$. Denote also their lengths by $\ell_n=\ell(I_n)$ and $\k_n=\nu(S_n)$. Furthermore, suppose that
$(\ell_n^{-1})_{n=1}^{\infty}$ and $(\kappa_n^{-1})_{n=1}^{\infty}$ have sub-exponential growth. Then we can write
\begin{align*}
 \pi(n,\alpha,I_n,S_n)=&\sum_{\tau\in\P_n}\mathbbm{1}_{I_n}\left(\log|\lambda(\tau)|-n\alpha\right)\mathbbm{1}_{S_n}\left( \hat \lambda(\tau)\right)\\
 =&\sum_{\tau\in\P_n}\mathbbm{1}_{\left[-\frac{1}{2},\frac{1}{2}\right]}\left(\ell_n^{-1}\left(\log|\lambda(\tau)|-n\alpha-p_n\right)\right)\mathbbm{1}_{\left[-\frac{\k}{2},\frac{\k}{2}\right]}\left(\frac{\k}{\k_n}\left(\hat \lambda(\tau) -\vartheta_n\right)\right)
\end{align*}

\subsection{Some auxiliary estimates}
\noindent We fix $\phi\in C^{4}(\R,\R_{\geq0})$ compactly supported and $\psi\in C^{4}(\mathbb S^1,\R_{\geq0})$ and consider the auxiliary counting number: \[ \pi_{\phi,\psi}(n):=\sum_{\tau\in\P_n}\phi\left(\ell^{-1}_n\left(\log|\lambda(\tau)|-n\alpha-p_n\right)\right)\psi\left(\frac{\k}{\k_n}\left(\hat \lambda(\tau)-\vartheta_n\right)\right).
\]

We study the asymptotic behaviour of $\pi_{\phi,\psi}$ to infer our result using an approximation argument in the next subsection.

We begin by changing the summation over $\P_n$, that is primitive periodic orbits of length $n$, to a sum over the set of fixed points of the iterated map $f^n$. Clearly, a primitive periodic orbit corresponds to $n$ distinct points in this set. However this set also contains points belonging in primitive periodic orbits of shorter lengths. In the following lemma we bound the error from these shorter primitive orbits.
\begin{lemma}\label{sumchange}For all $\eta>0$ we have that 
\begin{align*}
    \pi_{\phi,\psi}(n)=\frac{1}{n} \sum_{f^nx=x}\phi\left(\ell_n^{-1}\left(R^n(x)-p_n\right)\right)\psi\left(\frac{\k}{\k_n}(\theta^n(x)-\vartheta_n)\right)+O\left(e^{(H(\alpha)+\eta)n/2}\right)
\end{align*}
\end{lemma} 
\begin{proof}
Call a fixed point $x$ of the iterated map $f^n$ non-primitive when there exists $d$, a proper divisor of $n$, such that $f^dx=x$. We can then get the following bound,
\begin{flalign*}
\frac{1}{n}& \sum_{f^nx=x}\phi\left(\ell_n^{-1}\left(R^n(x)-p_n\right)\right)\psi\left(\frac{\k}{\k_n}(\theta^n(x)-\vartheta_n)\right)-\pi_{\phi,\psi}(n)\\
=\frac{1}{n}& \sum_{\substack{f^nx=x\\\text{non-primitive}}} \phi\left(\ell_n^{-1}\left(R^n(x)-p_n\right)\right)\psi\left(\frac{\k}{\k_n}(\theta^n(x)-\vartheta_n)\right)\\
=O&\Bigg(\frac{\norm{\psi}_\infty}{n} \sum_{\substack{d|n\\ d\leq n/2}} \sum_{f^dx=x}  \phi\left(\ell_n^{-1}(R^n(x)-p_n)\right)\Bigg)
=O\Bigg(\frac{1}{n} \sum_{d\leq n/2} \sum_{f^{d}x=x}  \frac{\phi\left(\ell_{n}^{-1}(R^n(x)-p_n\right)}{e^{\xi R^d(x)}} e^{\xi R^d(x)}\Bigg)
\end{flalign*}
We are only interested in periodic points which satisfy $\ell_n^{-1}\left(R^n(x)-p_n\right)\in\supp \phi$ that is $R^n(x)\in p_n+\ell_n \supp\phi$. 
Recalling that the intervals $I_n$ were chosen inside a compact set $K$ we conclude that for such a periodic point the absolute value of $R^n(x)$ is bounded. 
Therefore for a non-primitive periodic point $x$, satisfying $f^{d}x=x$ for $d$ as above, we get that $R^d(x)=\frac{d}{n}R^n(x)$ and thus $e^{\xi R^d(x)}$ is bounded from below. From this we conclude using Lemma \ref{decay} that for any $\eta>0$,
\begin{align*}
   &\frac{1}{n} \sum_{d\leq n/2} \sum_{f^{d}x=x}  \frac{\phi\left(\ell_{n}^{-1}(R^n(x)-p_n\right)}{e^{\xi R^d(x)}} e^{\xi R^d(x)} =O\Bigg(\frac{\norm{\phi}_\infty }{n} \sum_{d\leq n/2}  \sum_{f^dx=x}  e^{\xi R^d(x)}\Bigg)  =O\Bigg(\frac{\norm{\phi}_\infty }{n} \sum_{d\leq n/2} Z_d(\xi,0)\Bigg) \\
   =&O\Bigg(\frac{1}{n} \sum_{d\leq n/2} e^{\big(P(\xi R)+\eta\big)d}\Bigg)=O\left(e^{\big(H(\alpha)+\eta\big)n/2}\right).
\end{align*}

\end{proof}

Consequently, we define 
\begin{align}
    \tilde\pi_{\phi,\psi}(n)=\frac{1}{n} \sum_{f^nx=x}\phi\left(\ell_n^{-1}\left(R^n(x)-p_n\right)\right)\psi\left(\frac{\k}{\k_n}(\theta^n(x)-\vartheta_n)\right).
\end{align}

Setting
\[
\phi_{n}(x):=\phi(\ell_n\inv(x-p_n))e^{-\xi (x-p_n)}
\]
we note that $\phi_{n}\in C^4(\R,\R_{\geq0})$ and is also compactly supported. Similarly, set 
\[
\psi_n(x):=\psi\left(\frac{\k}{\k_n}(x-\vartheta_n)\right).
\] 
In this notation we have that \[
\tilde \pi_{\phi,\psi}=\frac{1}{n}\sum_{f^nx=x}\phi_n\left(R^n(x)\right)\psi_n\left(\theta^n(x)\right)e^{\xi \left(R^n(x)-p_n\right)}.
\]
\begin{proposition}\label{aux}
\[\tilde \pi_{\phi,\psi}\sim e^{-\xi p_n}\frac{\int \phi_n\,\int\psi_n}{\sigma_\alpha\sqrt{2\pi}}\frac{e^{H(\alpha)n}}{n^{3/2}}\;\;\;\;\text{as $n\ra\infty$}.
\]
\end{proposition}

To prove this proposition we consider \[
A(n):=\left | \frac{e^{\xi p_n}\sigma_\alpha\sqrt{2\pi n^3}}{ e^{H(\alpha)n}}\tilde \pi_{\phi,\psi} -\int_{\R} \phi_n\int_{\mathbb{S}^1} \psi_n \right |
\]
and show that $A(n)\rightarrow 0$ as $n\rightarrow\infty$. The following proposition provides us with an initial bound. Using Fourier inversion and Fourier expansion we get,
\begin{align}
    \phi_{n}(x)e^{\xi (x-p_n)}&= e^{-\xi p_n}\int_{-\infty}^\infty \hat\phi_{n}(t)e^{\left(\xi+2\pi it\right)x}\,dt\,\,\,\;\text{\; and \; }\;\,\,\label{fou1}\\
     \psi_n(x)&=\sum_{k\in\mathbb Z} c_{n,k}\, e^{2\pi ikx}.\label{fou2}
\end{align}

\begin{proposition}\label{a123}
\begin{align*}
    A(n)\leq \frac{1}{\sqrt{2\pi} } \int_{-\infty}^{\infty} \left | \sum_{k\in\mathbb Z} \frac{c_{n,k}}{e^{H(\alpha)n}} \, \hat\phi_{n}\left(\frac{t}{2\pi\sigma_\alpha\sqrt{n}}\right)Z_n\left(\xi+\frac{it}{\sigma_\alpha\sqrt{n}},k\right)- e^{-\frac{t^2}{2}}\int_{\R} \phi_{n}\int_{\mathbb{S}^1} \psi_n\right| \,dt .
\end{align*}
\end{proposition}
\begin{proof}
Using (\ref{fou1}) and (\ref{fou2}) we can get
\begin{align*}
        \frac{e^{\xi p_n}\sigma_\alpha\sqrt{2\pi n^3}}{ e^{H(\alpha)n}} \tilde\pi_{\phi,\psi}(n)=&\frac{\sigma_\alpha\sqrt{2\pi n}}{e^{H(\alpha)n}} \sum_{f^nx=x}  \int_{-\infty}^\infty \hat\phi_{n}(t)e^{(\xi+2\pi it)R^n(x)}\,dt\, \sum_{k\in\mathbb Z} c_{n,k}\, e^{2\pi ik\theta^n(x)}\\
    =& \frac{ 1}{\sqrt{2\pi}} \int_{-\infty}^{\infty}  \sum_{k\in\mathbb Z}  \frac{c_{n,k}}{ e^{H(\alpha)n}}    \hat\phi_{n}\left(\frac{t}{2\pi\sigma_\alpha\sqrt{n}}\right)\sum_{f^nx=x}e^{\left(\xi+\frac{it}{\sigma_\alpha\sqrt{n}}\right)R^n(x)+2\pi ik\theta^n(x)}\,dt\\
    =& \frac{ 1}{\sqrt{2\pi}} \int_{-\infty}^{\infty}  \sum_{k\in\mathbb Z}  \frac{c_{n,k}}{ e^{H(\alpha)n}} \, \hat\phi_{n}\left(\frac{t}{2\pi\sigma_\alpha\sqrt{n}}\right)Z_n\left(\xi+\frac{it}{\sigma_\alpha\sqrt{n}},k\right)\,dt.
\end{align*}
In addition, recalling that $\int_{-\infty}^\infty e^{-t^2/2}\, dt=\sqrt{2\pi}$ we get,
\begin{align*}
   \sqrt{2\pi}A(n)=&\left| \int_{-\infty}^{\infty} \sum_{k\in\mathbb Z}  \frac{c_{n,k}}{ e^{H(\alpha)n}} \, \hat\phi_{n}\left(\frac{t}{2\pi\sigma_\alpha\sqrt{n}}\right)Z_n\left(\xi+\frac{it}{\sigma_\alpha\sqrt{n}},k\right)- e^{-\frac{t^2}{2}}\int_{\R} \phi_{n}\int_{\mathbb{S}^1} \psi_n\,dt \right|\\
         \leq&  \int_{-\infty}^{\infty} \left | \sum_{k\in\mathbb Z} \frac{c_{n,k}}{ e^{H(\alpha)n}} \, \hat\phi_{n}\left(\frac{t}{2\pi\sigma_\alpha\sqrt{n}}\right)Z_n\left(\xi+\frac{it}{\sigma_\alpha\sqrt{n}},k\right)- e^{-\frac{t^2}{2}}\int_{\R} \phi_{n}\int_{\mathbb{S}^1} \psi_n\right| \,dt .
\end{align*}


\end{proof}

\noindent Consider now the following\begin{align*}
    &A_1(n):= \int_{-\e\sigma_\alpha\sqrt{n}}^{\e\sigma_\alpha\sqrt{n}} \left|  \sum_{k\in\mathbb Z} \frac{c_{n,k}}{ e^{H(\alpha)n}}  \,\hat\phi_{n}\left(\frac{t}{2\pi\sigma_\alpha\sqrt{n}}\right)Z_n\left(\xi+\frac{it}{\sigma_\alpha\sqrt{n}},k\right)-e^{-\frac{t^2}{2}}\int_{\R} \phi_{n}\int_{\mathbb{S}^1} \psi_n\right|\,dt,\\
     &A_2(n):= \int_{|t|\geq \e\sigma_\alpha\sqrt{n}}\left|  \sum_{k\in\mathbb Z} \frac{c_{n,k}}{ e^{H(\alpha)n}}  \,\hat\phi_{n}\left(\frac{t}{2\pi\sigma_\alpha\sqrt{n}}\right)Z_n\left(\xi+\frac{it}{\sigma_\alpha\sqrt{n}},k\right)\right|\,dt,\\
    &A_3(n):=  \int_{|t|\geq \e\sigma_\alpha\sqrt{n}}\left| e^{-\frac{t^2}{2}}\int_{\R} \phi_{n}\int_{\mathbb{S}^1} \psi_n \right| \,dt,
\end{align*}
with $\e>0$ small enough as in Lemmas \ref{pressure} and \ref{decay}. It then follows from Proposition \ref{a123} that 
\[A(n)\leq \frac{1}{\sqrt{2\pi} }\bigg[ A_1(n)+A_2(n)+A_3(n)\bigg].\]

\noindent We hence bound these three quantities separately to show that $\lim_{n\rightarrow\infty}A(n)=0$. To obtain these bounds we first recall a standard result from Fourier Analysis.

\begin{lemma}\label{CK}
If $\psi\in C^4(\mathbb S^1,\R)$ has Fourier coefficients $(c_k)_{k\in\mathbb Z}$ then $c_0=\int_{\mathbb S^1} \psi$ and uniformly for $\psi\in C^4(\mathbb S^1,\R)$ \[ c_k=O(\norm{\psi}_{C^4}|k|^{-4}).\] If $\phi\in C^4(\R,\R)$ is compactly supported and $\hat\phi$ is its Fourier transform then $\hat\phi(0)=\int_\R\phi$ and uniformly for $\phi\in C^4(\R,\R)$ we have that\[\hat \phi(u)=O(\norm{\phi}_{C^4}|u|^{-4}).\]
\end{lemma}

These bounds follow by repeated applications of integration by parts. 
Now since $$\psi_n^{(q)}(x)=\left(\frac{\k}{\k_n}\right)^{q} \psi^{(q)}\left(\frac{\k}{\k_n}(x-\vartheta_n)\right)$$ there exists a constant $C>0$ such that for all $n,|k|\geq 1$  
\begin{align}\label{fouc}
    |c_{n,k}| \leq C\k_n^{-4}|k|^{-4} \|\psi\|_{C^4}.
\end{align}
Similarly, there exists $C>0$ such that for $n\in\N$ and $u\in \R$
\begin{align}\label{fouc2}
    |\hat \phi_n(u)|\leq C\ell_n^{-4}|u|^{-4}\|\phi\|_{C^4}.
\end{align}

\begin{proposition}
$\lim_{n\ra\infty} A_1(n)=0.$
\end{proposition}
\begin{proof}

We can use inequality (\ref{zndecay}) to bound $Z_n\left(\xi+\frac{it}{\sigma_\alpha\sqrt{n}},k\right)$ for $k\neq 0$. Therefore fixing $\eta>0$ and recalling the bounds for the Fourier coefficients from (\ref{fouc}) we get  
\begin{align}\label{error1}
    &\int_{-\e\sigma_\alpha\sqrt{n}}^{\e\sigma_\alpha\sqrt{n}} \left| \sum_{k\neq 0}\frac{c_{n,k}}{e^{H(\alpha)n}}  \hat\phi_{n}\left(\frac{t}{2\pi\sigma_\alpha\sqrt{n}}\right)Z_n\left(\xi+\frac{it}{\sigma_\alpha\sqrt{n}},k\right)  \right|\,dt \nonumber\\ \nonumber
    =&  O\left(\int_{-\e\sigma_\alpha\sqrt{n}}^{\e\sigma_\alpha\sqrt{n}}\sum_{k\neq 0}   \k_n^{-4}|k|^{-4} \norm{\hat\phi_n}_\infty \left(\frac{|t|}{\sigma_\alpha\sqrt{n}}+|k|\right)^{2+\eta}\rho_\eta^n \,dt \right)=O\left( \k_n^{-4}\norm{\hat\phi_n}_\infty\, \rho_\eta^n  \right)
\end{align}  
for some $\rho_\eta\in(0,1)$. Since $\phi$ is compactly supported and $p_n\in K$ we can uniformly bound $\hat\phi_n$ for all $n\in\N$. Further, recalling that the sequence $(\k_n)_{n=1}^{\infty}$ is of sub-exponential growth we get that this error tends to zero as $n\ra\infty.$
Therefore, we are now left to bound 
\[
 \int_{-\e\sigma_\alpha\sqrt{n}}^{\e\sigma_\alpha\sqrt{n}} \left| \frac{c_{n,0}}{ e^{H(\alpha)n}}  \hat\phi_{n}\left(\frac{t}{2\pi\sigma_\alpha\sqrt{n}}\right)Z_n\left(\xi+\frac{it}{\sigma_\alpha\sqrt{n}},0\right) - e^{-\frac{t^2}{2}}\int_{\R} \phi_{n}\int_{\mathbb{S}^1}\psi_n \right|\,dt.
\]
Using part (2) from Lemma \ref{decay} we get that for some $\beta\in (0,1)$
\begin{align*}
 \int_{\mathbb{S}^1} \psi_n \int_{-\e\sigma_\alpha\sqrt{n}}^{\e\sigma_\alpha\sqrt{n}} \left| \hat\phi_{n}\left(\frac{t}{2\pi\sigma_\alpha\sqrt{n}}\right) e^{n\left(P\left(\left(\xi+\frac{it}{\sigma_\alpha\sqrt{n}}\right)R\right)-H(\alpha)\right)} - e^{-\frac{t^2}{2}}\int_{\R} \phi_{n}\right|\,dt+O\left(\beta^n\right),
\end{align*}
 On the domain of integration, we see that as $n\rightarrow\infty$
\begin{enumerate}
    \item $e^{n\left(P\left(\left(\xi+\frac{it}{\sigma_\alpha\sqrt{n}}\right)R\right)-H(\alpha)\right)}\rightarrow e^{-t^2/2}$ by Lemma \ref{pressure},
    \item $\hat\phi_{n}\left(\frac{t}{2\pi\sigma_\alpha\sqrt{n}}\right)\ra \hat\phi_n(0)=\int_\R \phi_{n} $ by continuity.
\end{enumerate}
\noindent Furthermore, for large $n$ we have the bound $e^{n\left(P\left(\left(\xi+\frac{it}{\sigma_\alpha\sqrt{n}}\right)R\right)-H(\alpha)\right)}\leq e^{-t^2/4}$ and so \[\left|e^{n\left(P\left(\left(\xi+\frac{it}{\sigma_\alpha\sqrt{n}}\right)R\right)-H(\alpha)\right)}-e^{-t^2/2}\right|\leq 2e^{-t^2/4}.
\]
Finally, since $\hat\phi_n$ is uniformly bounded, we can apply the Dominated Convergence Theorem to get that $\lim_{n\rightarrow\infty}A_1(n)=0$.

\end{proof} 
\begin{proposition}
$\lim_{n\ra\infty}A_2(n)=0.$
\end{proposition}
\begin{proof}
 \[A_2(n)\leq \sum_{k\in\mathbb Z}  \frac{|c_{n,k}|}{ e^{H(\alpha)n}}  \int_{|t|\geq \e\sigma_\alpha\sqrt{n}} \left| \hat\phi_{n}\left(\frac{t}{2\pi\sigma_\alpha\sqrt{n}}\right) Z_n\left(\xi+\frac{it}{\sigma_\alpha\sqrt{n}},k\right)\right|\,dt.
\]

\noindent Firstly, we use the bounds from (\ref{fouc}) and (\ref{fouc2}). In addition, for $k\neq 0$ we use inequality (\ref{zndecay}) to get that a fixed $\eta\in (0,1)$ there exists $\rho_{\eta}\in (0,1)$ such that
\begin{align*}
    &\sum_{k\neq0}  \frac{|c_{n,k}|}{ e^{H(\alpha)n}}  \int_{|t|\geq \e\sigma_\alpha\sqrt{n}} \left| \hat\phi_{n}\left(\frac{t}{2\pi\sigma_\alpha\sqrt{n}}\right) Z_n\left(\xi+\frac{it}{\sigma_\alpha\sqrt{n}},k\right)\right|\,dt\\
    =&O\left( \sum_{k\neq 0} \k_n^{-4}|k|^{-4} \int_{|t|\geq \e\sigma_\alpha\sqrt{n}} \left| \ell_n^{-4}\left(\frac{t}{2\pi\sigma_\alpha\sqrt{n}}\right)^{-4} \left(\left|\frac{t}{\sigma_\alpha\sqrt{n}}\right|+|k|\right)^{2+\eta} \rho_\eta^n\right|\,dt \right)\\
    =&O\left(\frac{n^2\rho_\eta^n}{\k_n^{4}\ell_n^{4}}\sum_{k\neq0}  \int_{|t|\geq \e\sigma_\alpha\sqrt{n}} \frac{\big(\left|t/\sigma_\alpha\sqrt{n}\right|+|k|\big)^{2+\eta}}{ t^{4}k^{4} } \,dt \right)=O\left(\frac{n^2}{\k_n^{4}\ell_n^{4}}\rho_\eta^n\right).
\end{align*}
\noindent On the other hand, for $k=0$ we get using part (2) of Lemma \ref{decay} that for some $\beta\in(0,1)$,
\[   \frac{|c_{n,0}|}{ e^{H(\alpha)n}}  \int_{|t|\geq \e\sigma_\alpha\sqrt{n}}^{|t|\leq \sigma_\alpha\sqrt{n}} \left| \hat\phi_{n}\left(\frac{t}{2\pi\sigma_\alpha\sqrt{n}}\right) Z_n\left(\xi+\frac{it}{\sigma_\alpha\sqrt{n}},0\right)\right|\,dt\\=O\left(|c_{n,0}|\norm{\hat\phi_{n} }_{\infty} \beta^{n}\right)=O(|c_{n,0}|\beta^{n}),
\]since $\hat\phi_n$ is uniformly bounded across all $n\in\N.$ We can also uniformly bound $|c_{n,0}|$ since \[c_{n,0}=\int_{\mathbb S^1 }\psi_n\leq\|\psi\|_\infty.\]

Finally, as above we can use inequality (\ref{zndecay}) to bound the rest by the following,
\begin{align*}
&\frac{|c_{n,0}|}{ e^{H(\alpha)n}}   \int_{|t|\geq \sigma_\alpha\sqrt{n}} \left| \hat\phi_{n}\left(\frac{t}{2\pi\sigma_\alpha\sqrt{n}}\right) Z_n\left(\xi+\frac{it}{\sigma_\alpha\sqrt{n}},0\right)\right|\,dt\\
=&O\left( \int_{|t|\geq \sigma_\alpha\sqrt{n}} \ell_n^{-4}\left(\frac{t}{2\pi\sigma_\alpha\sqrt{n}}\right)^{-4}  \left|\frac{t}{\sigma_\alpha\sqrt{n}}\right|^{2+\eta} \rho_\eta^n \,dt \right)\\
=&O\left(\frac{n^2\rho_\eta^n}{\ell_n^{4} }  \int_{|t|\geq \sigma_\alpha\sqrt{n}}   |t|^{\eta-2}  \,dt \right)=O\left(\frac{n^2}{\ell_n^{4} }\rho_\eta^n  \right).
\end{align*}

\noindent Combining the three bounds obtained above and recalling that the sequences $(\ell_{n})_{n=1}^{\infty}$ and $(\k_{n})_{n=1}^{\infty}$ are of sub-exponential growth we obtain that $\lim_{n\rightarrow\infty} A_{2}(n)=0$.

\end{proof}
Finally, it is clear that $\lim_{n\ra\infty}A_3(n)=0$. This completes the proof of Proposition \ref{aux}.

\subsection{Approximation argument}

Here we show how the previous auxiliary estimates provide us with the proof of Theorem \ref{main} through an approximation argument. By Proposition \ref{aux} and Lemma \ref{sumchange} we have that for all compactly supported $\phi\in C^4(\mathbb R,\R)$ and all $\psi\in C^4(\mathbb S^1,\R)$ \begin{align}\label{auxlim}
  \pi_{\phi,\psi}(n)\sim e^{-\xi p_n}\frac{\int \phi_n\,\int\psi_n}{\sigma_\alpha\sqrt{2\pi}}\frac{e^{H(\alpha)n}}{n^{3/2}}
\end{align}
as $n\ra\infty.$

Fixing $\eta>0$ we wish to construct compactly supported $\phi \in C^{4}(\R,\R)$ and $\psi \in C^{4}(\mathbb S^1,\R)$ satisfying the following:
\begin{align*}
&\mathbbm{1}_{\left[-\frac{1}{2},\frac{1}{2}\right]}\leq\phi \leq 1+\eta, \;\; \supp(\phi )\subset \left[-\frac{1+\eta}{2},\frac{1+\eta}{2}\right] \;\text{ and }\; \int_\R \phi \leq 1+\eta,\\
&\mathbbm{1}_{\left[-\frac{\k}{2},\frac{\k}{2}\right]}\leq\psi \leq 1+\eta, \;\; \supp(\psi )\subset \left[-\frac{\k+\eta}{2},\frac{\k+\eta}{2}\right] \;\text{ and }\; \int_{\mathbb S^1} \psi \leq \k+\eta.
\end{align*}
A smooth function $\Phi:\R\ra\R_{\geq0}$ is called a positive mollifier, if it satisfies the following properties:
\begin{enumerate}
    \item it is compactly supported,
    \item $\int_\R \Phi =1$,
    \item $\lim_{\e\ra0} \Phi_\e(x)=\lim_{\e\ra0} \e\inv\Phi(x/\e)=\delta(x)$ where $\delta(x)$ is the Dirac delta function.
\end{enumerate}

Let $\gamma_1,...,\gamma_4>0$ and set $G=(1+\gamma_1)\mathbbm{1}_{\left[-\frac{1}{2}-\gamma_2,\frac{1}{2}+\gamma_2\right]}$ and $H=(1+\gamma_3)\mathbbm1_{\left[-\frac{\k}{2}-\gamma_4,\frac{\k}{2}+\gamma_4\right]}$. Then for sufficiently small $\e,\gamma_1,...,\gamma_4>0$ the functions \[
\phi =G*\Phi_\e  \;\;\text{  and  }\;\;  \psi =H*\Phi_e
\]
satisfy all the required properties. Note that since $\k<1$ and the constants $\e,\,\gamma_4$ were chosen sufficiently small it is harmless to assume that $\psi$ is defined on $\R$ rather than $\mathbb S^1$. Using (\ref{auxlim}) and the properties above we can deduce that
\begin{align*}
\limsup_{n\rightarrow\infty}& \frac{\sigma_\alpha\sqrt{2\pi n^3}}{e^{H(\alpha)n}} \pi(n,\alpha,I_n,S_n)\\
=\limsup_{n\rightarrow\infty}& \frac{\sigma_\alpha\sqrt{2\pi n^3}}{e^{H(\alpha)n}} \sum_{\tau\in\P_n}\mathbbm{1}_{\left[-\frac{1}{2},\frac{1}{2}\right]}\left(\ell^{-1}_{n}(\log|\lambda(\tau)|-n\alpha-p_n)\right)\mathbbm{1}_{\left[-\frac{\k}{2},\frac{\k}{2}\right]}\left(\frac{\k}{\k_n}(\hat \lambda(\tau)-\vartheta_n)\right)\\
\leq\limsup_{n\rightarrow\infty} &\frac{\sigma_\alpha\sqrt{2\pi n^3}}{e^{H(\alpha)n}} \sum_{\tau\in\P_n}\phi\left(\ell^{-1}_{n}(\log|\lambda(\tau)|-n\alpha-p_n)\right)\psi\left(\frac{\k}{\k_n}(\hat \lambda(\tau)-\vartheta_n)\right)\\
=\limsup_{n\rightarrow\infty} & \,e^{-\xi p_n}\int_\R\phi_n \int_{\mathbb S^1}\psi_n.
\end{align*}
We have
\begin{align*}
    &\int_{\mathbb R}\psi_n=\int_\R \psi\left(\frac{\k}{\k_n}(y-\vartheta_n)\right)\,dy=\int_\R \psi\left(\frac{\k}{\k_n}y\right)\,dy=\frac{\k_n}{\k}\int_\R\psi\leq\k_n+\frac{\eta}{\kappa}=\nu(S_n)+O(\eta).
\end{align*}
Similarly,
\begin{align*}
&\int_\R \phi_n=\int_\R \phi(\ell_n\inv(x-p_n))e^{-\xi(x-p_n)}\,dx=\ell_n\int_\R \phi(u)e^{-\xi\ell_n u}\,du\\
=\ell_n&\int_{[-\frac{1+\eta}{2},\frac{1+\eta}{2}]} \phi(u)e^{-\xi\ell_n u}\,du\leq \ell_n\int_{[-\frac{1}{2},\frac{1}{2}]} \phi(u)e^{-\xi\ell_n u}\,du+\eta(1+\eta) e^{2(1+|\xi|)|K|}|K|\\\\
&\ell_n\int_{[-\frac{1}{2},\frac{1}{2}]} \phi(u)e^{-\xi\ell_n u}\,du \leq\ell_n\int_{[-\frac{1}{2},\frac{1}{2}]} (1+\eta)e^{-\xi\ell_n u}\,du\leq e^{\xi p_n}\int_{I_n} e^{-\xi u}\,du+\eta e^{(1+|\xi|)|K|}|K|
\end{align*}
Therefore, $$e^{-\xi p_n}\int_\R \phi_n\int_{\mathbb S^1}\psi_n\leq\nu(S_n)\int_{I_n} e^{-\xi u}\,du+O(\eta).$$

\noindent Similarly, one can show that
\begin{align*}
    \liminf_{n\rightarrow\infty}& \frac{\sigma_\alpha\sqrt{2\pi n^3}}{e^{H(\alpha)n}} \pi(n,\alpha,I_n,S_n)\geq \liminf_{n\rightarrow\infty}  \left(\nu(S_n)\int_{I_n} e^{-\xi x}\,dx\right)+O(\eta).
\end{align*}
Since the choice of $\eta>0 $ was arbitrary we get the result.


Assuming $\lim_{n\ra\infty} \ell_n=0$ the derivation of the asymptotic formula (\ref{main2}) from (\ref{main1}) is straightforward. The asymptotic formula (\ref{main3}) corresponding to choosing the measure of maximal entropy follows in a similar manner. By the definition of the pressure function $\mu_{\max}$ is the equilibrium state of $\xi R$ for $\xi=0$. Then, the proof follows in the same way as above.


\end{document}